\documentclass[12pt]{amsart}
\usepackage{amssymb}
\usepackage[T1]{fontenc}
\usepackage{graphicx}

\usepackage{stmaryrd} 
\usepackage{dsfont} 
\usepackage{mathtools} 
\usepackage{enumitem}
\usepackage{graphicx}
\usepackage{tikz-cd}

\theoremstyle{plain}
\newtheorem{theorem}{Theorem}[section]
\newtheorem{lemma}[theorem]{Lemma}
\newtheorem{proposition}[theorem]{Proposition}
\newtheorem{corollary}[theorem]{Corollary}
\theoremstyle{definition}
\newtheorem{definition}[theorem]{Definition}

\newtheorem{example}[theorem]{Example}

\numberwithin{equation}{section}

\renewcommand{\epsilon}{\varepsilon}
\renewcommand{\phi}{\varphi}
\newcommand{\A}{\mathcal{A}}
\newcommand{\F}{\mathcal{F}}
\newcommand{\M}{\mathcal{M}}
\newcommand{\Mna}{\mathcal{M}_{\text{n.a.}}}
\newcommand{\X}{\mathcal{X}}
\newcommand{\I}{\mathcal{I}}
\renewcommand{\P}{\mathcal{P}}
\DeclareMathOperator{\Acc}{Acc}
\DeclareMathOperator{\len}{len}
\newcommand{\Leb}{\lambda}
\newcommand{\hfair}{h_{\textnormal{fair}}}

\begin{document}

\title{Fair Measures for Countable-to-one Maps}
\author{Ana Rodrigues}
\address{Department of Mathematics\\University of Exeter\\Exeter EX4 4QF, UK}
\email{A.Rodrigues@exeter.ac.uk}
\author{Samuel Roth}
\address{Mathematical Institute of the Silesian University in Opava\\Na Rybni\v{c}ku 1\\74601 Opava, Czech Republic}
\email{samuel.roth@math.slu.cz}
\author{Zuzana Roth}
\address{Mathematical Institute of the Silesian University in Opava\\Na Rybni\v{c}ku 1\\74601 Opava, Czech Republic}
\email{zuzana.roth@math.slu.cz}
\thanks{Research was supported by SGS 16/2016 and by RVO funding for I\v{C}47813059.}
\begin{abstract}

In this paper we generalize the recently introduced concept of fair measure (M.~Misiurewicz and A.~Rodrigues, Counting preimages. \emph{Ergod. Th. \& Dynam. Sys.}
\textbf{38} (2018), no. 5, 1837 -- 1856). We study transitive countable state Markov shift maps and extend our results to a particular class of interval maps, Markov and mixing interval maps. Finally, we move beyond the interval and look for fair measures for graph maps.

\end{abstract}
\maketitle

\section{Introduction}

Computing the topological entropy of a one-dimensional dynamical system is in general a very difficult task. Motivated by this problem, in \cite{MR}, the entropy was computed  following backward trajectories in a way that at each step every preimage can be chosen with equal probability introducing a new concept of entropy, the fair entropy. Fair entropy gives a lower bound for topological entropy and is simple to compute.

As in \cite{MR} let us denote by $c(x)$ the cardinality of the set $f^{-1}(x)$. We start with a point $x_0$ and proceed by induction. Given $x_n$, we choose $x_{n+1}$ from the set $f^{-1}(x_n)$ randomly, that is, the probability of choosing any of these points is $1/c(x_n)$. Then we go to the limit with the geometric averages of $c(x_0), c(x_1), \ldots, c(x_n)$ as $n$ goes to infinity. 

Also in \cite{MR}, convergence of the geometric averages of $c(x_0), c(x_1), \ldots, c(x_n)$ as  $n$ goes to infinity for a random choice of the backward trajectory as well as convergence of the measures equidistributed along longer and longer pieces of a random backward trajectory were investigated. Indeed, these questions were answered for some special classes of maps, namely, transitive subshifts of finite type and  piecewise monotone (with a finite number of pieces) topologically mixing interval maps.

In this paper we define fair measures in a broad setting, flexible enough to handle noncompact spaces and discontinuous maps.

Let us state our main definition. Let $(X,\F)$ be a measurable space and $f:X\to X$ a surjection. Assume that $X$ admits a countable measurable partition $\X=\{X_i\}_{i=1}^{\infty}$ so that each $f(X_i)$ is also measurable and each restriction $f|_{X_i} : X_i \to f(X_i)$ is a measurable isomorphism, i.e. images and preimages of measurable sets are again measurable. Let $\M(X,f)$ be the space of all $f$-invariant probability measures on $X$. To avoid pathologies, we impose the mild topological assumption that $X$ is a Polish space and $\F$ is the Borel $\sigma$-algebra. Then for each $\mu\in\M(X,f)$, the measure-theoretic completion of $(X,\F,\mu)$ is a standard probability space (Lebesgue space). 

We define another countable measurable partition $\A$ as the common refinement of the partitions $\left\{f(X_i),X\setminus f(X_i)\right\}$, $i=1,2,\ldots$. Thus, if a set $B$ is a subset of an element of $\A$ (in symbols, $B\prec\A$), then we know which branches of $f^{-1}$ are defined on $B$. We write $p(B)=\{i; B\subset f(X_i)\}$ to identify those branches and $c(B)=\# p(B)$ to count how many there are. A singleton $\{x\}$ is always contained in an element of $\A$ and we write simply $c(x)$ for the number of preimages of $x$. This number is always positive since $f$ is surjective, but may be infinite.

Now we define a fair measure as a special kind of invariant measure for which the measure of a set is divided equally among the pieces of its preimage.

\begin{definition}\label{defnFair} An invariant measure $\mu\in\M(X,f)$ is called \emph{fair} if each measurable set $B\prec \mathcal{A}$ satisfies
\begin{equation}\label{fair}
\mu(X_i\cap f^{-1}(B)) = \frac{\mu(B)}{c(B)} \text{, for all } i\in p(B).
\end{equation}
\end{definition}
By a simple common refinement argument, the definition of a fair measure does not depend on the choice of the partition $\X$.

With our definition we are able to study the behaviour of a random backward trajectory for  transitive countable state Markov shift maps. We then extend the results from shift maps to maps on the interval. To do this, we introduce the notion of {\it isomorphism modulo countable invariant sets} and we then study Markov and mixing interval maps.

This paper is organized as follows. In Section 2 we provide some more definitions and describe the general setting. In Section 3 we discuss random backward trajectories for transitive countable state Markov shift maps and in Section 5 we introduce the main tool that will allow us to extend our results to countably Markov and mixing interval maps in Section 6. We finally introduce Lebesgue fair models in Section 7 and in Section 8  we look for fair measures on graph maps.

\section{General Case}

In this section we study some properties of fair measures. Immediately from our Definition \ref{defnFair} we get the following properties of a fair measure.

\begin{lemma}\label{lem:basic}
If $\mu$ is a fair measure, then

(a) For a measurable set $B\prec\A$ with $c(B)=\infty$ we have $\mu(B)=0$.

(b) If $B$ is measurable and $\mu(B)=0$, then $\mu(f(B))=0$.

(c) $\mu(\{x\in X ; \#f^{-n}(x)=\infty \text{ for some }n\in\mathbb{N}\})=0$.

\end{lemma}
\begin{proof}
For $B\prec\A$ we use invariance of the measure to write $\mu(B)=\sum \mu(X_i\cap f^{-1}(B))$ where the sum extends over all $i\in p(B)$. By~\eqref{fair} each summand is zero. This proves (a).

By cutting up a given measurable set $B$ into pieces, we may assume that $B\prec\X$ and $f(B)\prec\A$. If $B\subset X_i$, then $B = X_i\cap f^{-1}(f(B))$. Then by~\eqref{fair}, $\mu(f(B))=c(f(B))\cdot\mu(B)$, which is zero if $c(f(B))$ is finite. But if $c(f(B))$ is infinite, then $\mu(f(B))=0$ by (a). This completes the proof of (b).

To prove (c) note that this set can be written as the union of the sets $f^n(A)$ where $A\in\A$, $c(A)=\infty$, and $n\geq0$. By (a) if $c(A)=\infty$, then $\mu(A)=0$. Now the result follows from (b).
\end{proof}

Thus for maps like the Gauss map, where every point has countably many preimages, we have no hope of finding a fair measure. In the search for fair measures, we must focus our attention on the part of the space where the cardinalities of preimage sets are finite.

In~\cite{MR} if a system $(X,f)$ has a unique fair measure, then the measure-theoretic entropy of that measure is referred to as the \emph{fair entropy} of the system. For systems with more than one fair measure, we generalize that definition as follows:

\begin{definition}\label{def:fairentropy}
The \emph{fair entropy} of a system $(X,f)$ is the supremum of measure-theoretic entropies of its fair measures, $\hfair(f)=\sup\left\{h_{\mu}(f)~|~\mu\in\M(X,f)\text{ is fair}\right\}$.
\end{definition}

It may also happen that a system has no fair measures, as in Examples~\ref{ex:null} and~\ref{ex:transient} below. In this case, we take the supremum over the empty set in Definition~\ref{def:fairentropy} to be zero.

%

Next we record two properties of fair measures which were proved in a more restrictive setting in~\cite{MR}. But the proofs need no modification; the arguments involved are purely measure-theoretic and do not require compactness of $X$, continuity of $f$, or finiteness of $\X$.

\begin{lemma}[\cite{MR}]\label{lem:Jac} A measure $\mu\in\M(X,f)$ is fair if and only if its measure-theoretic Jacobian is $x\mapsto c(f(x))$.
\end{lemma}

\begin{lemma}[\cite{MR}]\label{lem:erg-dec}If $\mu$ is a fair measure, then so is almost every component of its ergodic decomposition.
\end{lemma}

The \emph{Jacobian} of $\mu\in\M(X,f)$ referred to in Lemma~\ref{lem:Jac} is the measurable function $J:X\to[0,\infty)$, unique up to changes on sets of $\mu$-measure zero, such that for every Borel set $B\subset X$,
\begin{equation}\label{Jac}
\text{If } f|_B \text{ is injective, then } \mu(f(B))=\int_B J \, d\mu.
\end{equation}

In particular, for the Jacobian to exist, the measure $\mu$ must be \emph{non-singular}, i.e. every measure-zero set must have a measure zero image. Fair measures always fulfill this condition -- Lemma~\ref{lem:basic}~(b) -- and their Jacobians always exist, see~\cite[Proposition 9.7.2]{VO}.

\section{Random Backward Trajectories}

The motivation for studying fair measures is to understand what happens along a random backward trajectory $y_0 \mapsfrom y_1 \mapsfrom y_2 \mapsfrom \cdots$ of a given point $y_0\in X$, where the backward trajectory is chosen as follows: given $y_i$, we choose $y_{i+1}$ from $f^{-1}(y_i)$ by ``rolling a dice'' with $c(y_i)$ sides. Thus, at each stage of the process, the choice of the next preimage is ``fair.'' We hope that by distributing point masses along longer and longer pieces of this backward trajectory and passing to a weak-* limit
we can generate a fair measure. And we hope that we can calculate the entropy of this measure by taking geometric averages of the function $c$ along longer and longer pieces of the orbit.

One way to formalize what we mean by a random backward trajectory of a point $y_0$ is as follows.

\begin{definition}\label{def:rbt} Let $y_0\in X$ satisfy $\# f^{-n}(y_0)<\infty$ for all $n\geq0$. Consider the Markov sequence of random variables $Y_0, Y_1, \ldots$ with
\begin{itemize}
\item initial distribution $P[Y_0=y_0]=1$, and
\item transition probabilities $P[Y_{i+1}=y_{i+1} \, | \, Y_i=y_i] = 1/c(y_i)$ for each $y_{i+1}\in f^{-1}(y_i)$.
\end{itemize}
A property is said to hold for a \emph{random choice of the backward trajectory} of $y_0$ if the property holds for almost every outcome $(y_i)_{i=0}^\infty$ of this Markov chain.
\end{definition}

Note that we cannot discuss random backward trajectories of points $y_0$ for which $\# f^{-n}(y_0)=\infty$ for some $n\geq0$. When forced to choose among infinitely many preimages, there is no fair way to do it. In light of Lemma~\ref{lem:basic}~(c) this does not bother us too much.

Next, we state clearly the meaning of weak-* convergence, remembering that our space $X$ need not be compact. Let $\mu_n, \mu$ be Borel probability measures on $X$. We say that $\mu$ is the \emph{weak-* limit} of the measures $\mu_n$ and write $\mu_n \xrightarrow{weak-*} \mu$ if one (all) of the following equivalent conditions is satisfied (see \cite[Proposition 2.7]{DGS}):

\vspace{0.1in}

(a)  $\int_X \phi \, d\mu_n \to \int_X \phi \, d\mu$ for all bounded (!) continuous functions $\phi:X\to\mathbb{R}$.

(b) $\limsup \mu_n(C) \leq \mu(C)$ for every closed subset $C\subset X$.

(c) $\liminf \mu_n(U) \geq \mu(U)$ for every open subset $U\subset X$.

(d) $\lim \mu_n(B) = \mu(B)$ for every subset $B\subset X$ whose boundary has measure $\mu(\partial B)=0$.

\vspace{0.1in}

Finally, we say that a sequence of points $y_n\in X$ \emph{equidistributes} for the measure $\mu$ if $\mu$ is the weak-* limit of the measures $\frac{1}{N}\sum_{n=0}^{N-1} \delta_{y_n}$.

Just as ergodic invariant measures can be used to understand the behavior of typical forward trajectories of a system, so also ergodic fair measures give us information about typical backward trajectories. This is the meaning of the following theorem -- it is a straightforward adaptation of~\cite[Theorem 3.4]{MR}.

\begin{theorem}[\cite{MR}]\label{th:back}
Let $\mu$ be an ergodic fair measure. Then:

(a) For each integrable function $\phi:X\to\mathbb{R}$, for $\mu$-almost every $y_0\in X$, for a random choice $(y_n)$ of the backward trajectory of $y_0$, $\frac{1}{N}\sum_{n=0}^{N-1} \phi(y_n) \to \int_X \phi\, d\mu$.

(b) If $X$ is compact, then for $\mu$-almost every $y_0\in X$, a random choice $(y_n)$ of the backward trajectory of $y_0$ equidistributes for the measure $\mu$.

(c) For $\mu$-almost every $y_0\in X$, for a random choice $(y_n)$ of the backward trajectory of $y_0$, the geometric averages $\sqrt[n]{c(y_0) \cdot c(y_1) \cdot \ldots \cdot c(y_n)}$ converge as $n\to\infty$ to the (possibly infinite) number $\exp\left(\int\log c\, d\mu\right)$.

(d)  If additionally $f$ has a one-sided generator of finite entropy, then the limit in~(c) is the exponential of the entropy of $\mu$.

\end{theorem}

\begin{proof}[Sketch of proof]

(a) Form the natural extension $(\tilde{X}, \tilde{f}, \tilde{\mu})$ of $(X,f,\mu)$. Then apply Birkhoff's ergodic theorem using $\tilde{f}^{-1}$. For more details, see~\cite[Theorem 3.4]{MR}.

(b) Let $C(X)$ be the space of continuous real-valued functions on $X$ with the topology of uniform convergence. By compactness of $X$, $C(X)$ contains a countable dense subset $\{\phi_i\}_{i=1}^\infty$, see~\cite{PU}. Applying (a) to each $\phi_i$, we find a full-measure set of points $y_0$ such that $\frac1N \sum_{n=0}^{N-1} \phi_i(y_n) \to \int_X \phi_i \, d\mu$ for a random backward trajectory of $y_0$. If we choose $y_0$ from the intersection of these countably many full-measure sets, then a random backward trajectory will equidistribute for the measure $\mu$.

(c) If $\phi(x)=\log(c(x))$ is integrable, then apply~(a) directly. Otherwise, approximate $\phi$ from below by truncations $\phi_m(x)=\min(\phi(x),m)$ and apply~(a) anyway.

(d) This is Rohlin's entropy formula $h_\mu(f)=\int \log J \, d\mu$ for systems with a one-sided generator \cite[Theorem 2.9.7]{PU}, together with the observation that $\int \log c\circ f\, d\mu = \int \log c \, d\mu$ by the invariance of $\mu$.
\end{proof}

Theorem~\ref{th:back} gives us only partial information about random backward trajectories. One problem is that we do not know if there are any fair measures to apply it to. Another problem is that the results hold only almost everywhere, which may mean almost nowhere with respect to some other natural measure. The situation is much better, however, for countable state Markov shifts.

\section{Countable State Markov Shifts}\label{sec:cms}

Let $\I$ be a countable set of indices (states) and $M=(m_{ij})_{i,j\in \I}$ a 0-1 matrix. Consider the corresponding one-sided Markov shift $(\Sigma_M,\sigma)$. We assume transitivity of $\sigma$, which is the same as irreducibility of $M$. The number of preimages $c(x)$ of a point $x\in\Sigma_M$ depends only on the cylinder of length 1 to which it belongs, and if it is the $j$th cylinder, then it is equal to
\begin{equation}\label{cj}
c_j=\sum_{i} m_{ij}.
\end{equation}

If one of the column sums $c_j$ is infinite, then by Lemma~\ref{lem:basic} (c) every fair measure assigns the value zero to the set $\bigcup_{n=0}^\infty \sigma^n([j])$. But by transitivity of our Markov shift, this union is the whole space $\Sigma_M$. We may conclude that there are no fair measures. 

Assume from now on that all the column sums $c_j$ are finite. One natural way to look for a fair measure is to use stochastic Markov chains. Form a matrix $Q$ with entries
\begin{equation}\label{Q}
q_{ji}=\frac{m_{ij}}{c_j}.
\end{equation}
It is a nonnegative matrix with rows summing to 1 and nonzero entries in the same positions as the transpose matrix $M^T$. So we can ask whether the time-reversed Markov shift $\Sigma_{M^T}$ supports a shift-invariant Markov measure with transition probabilities $Q$. What is needed is a vector $\pi$ satisfying 
\begin{equation}\label{pi}
\pi Q=\pi,\qquad \text{with each }\pi_i\geq0 \text{ and }\sum_{i\in\I} \pi_i=1,
\end{equation}
to serve as the initial distribution. The theory of stochastic Markov chains tells us that there is at most one such vector $\pi$, its entries are necessarily strictly positive, and it exists if and only if $Q$ is positive recurrent~\cite{Fel}. In this case, we can use $\pi$ to construct a stochastic matrix $P$ with entries given by
\begin{equation}\label{P}
\pi_i p_{ij} = \pi_j q_{ji}.
\end{equation}
Summing~\eqref{P} over $j$ we get that $\pi P = \pi$, so that the measure $\mu=\text{Markov}(\pi, P)$ is shift-invariant. $\mu$ defines a measure on $\Sigma_M$ because $P, M$ have their nonzero entries in the same positions. To show that $\mu$ is fair, it suffices to check equation~\eqref{fair} on cylinder sets. Let $B=[j j_1 j_2 \cdots j_n]$ be a cylinder set. We are using $\X=\{[i]; i\in \I\}$. Then $p(B)=\{i; m_{ij}=1\}$ and $c(B)=c_j$. For each $i\in p(B)$ we have $[i]\cap \sigma^{-1}(B) = [i j j_1 \cdots j_n]$. The measure of this set is
$\pi_i p_{ij} p_{jj_1} \cdots p_{j_{n-1}j_n}$, which by~\eqref{P} equals $q_{ji}\pi_j p_{jj_1} \cdots p_{j_{n-1}j_n}$, which by~\eqref{Q} equals $\frac{1}{c(B)}\mu(B)$. This shows that $\mu$ is fair.

\begin{theorem}\label{th:cms}
Let $(\Sigma_M,\sigma)$ be a transitive countable-state Markov shift with all $c_j$ finite. Given any point $y_0\in\Sigma_M$ the behavior of a random backward trajectory $(y_n)$ is as follows:

(a) If $Q$ is positive recurrent, then $(y_n)$ equidistributes for the fair measure $\mu=\textnormal{Markov}(\pi, P)$.

(b) If $Q$ is null recurrent, then $(y_n)$ is dense in $\Sigma_M$, but visits each cylinder set $[i]$ with limiting frequency zero.

(c) If $Q$ is transient, then $(y_n)$ visits each cylinder set $[i]$ only finitely often.

\end{theorem}

\begin{proof}
Write $y_0=\omega_0\omega_1\omega_2\cdots$. Let $\delta_{\omega_0}$ be the probability vector on the state space $\I$ with a $1$ in position $\omega_0$ and $0$'s elsewhere. Consider the measure $\nu_0=\text{Markov}(\delta_{\omega_0},Q)$ on the space $\Sigma_{M^T}$. In the positive recurrent case there is also a stationary probability vector $\pi$ for $Q$ and the corresponding measure $\nu=\text{Markov}(\pi,Q)$ on $\Sigma_{M^T}$.

We want to choose a backward trajectory for $y_0$ in a fair way. Any point $\omega'_0 \omega'_1 \cdots \in \Sigma_{M^T}$ with $\omega'_0=\omega_0$ can be used as a possible history. One way to think of this is that we extend our one-sided sequence $y_0$ to a two-sided sequence
\begin{equation*}
\lefteqn{\underbrace{\phantom{\cdots\, \omega'_3\, \omega'_2\, \omega'_1\, \omega'_0\,}}_{\text{taken from } \Sigma_{M_T}}} \cdots \omega'_3\, \omega'_2\, \omega'_1\, \!\overbrace{\omega'_0\, \omega_1\, \omega_2\, \omega_3\, \cdots}^{\text{initial point }y_0}.
\end{equation*}
Then the backward trajectory consists of the points $y_n = \omega'_n \omega'_{n-1} \cdots \omega'_0 \omega_1 \omega_2 \cdots$. To make this choice in a fair way, we use the Markov chain $(\Sigma_{M_T}, \nu_0)$. We define a sequence of random variables $Y_0, Y_1, \ldots$ on the probability space $(\Sigma_{M^T},\nu_0)$ by
\begin{equation}\label{recover}
Y_n(\omega'_0\omega'_1\omega'_2\cdots) = \omega'_n\omega'_{n-1}\cdots\omega'_1\omega'_0\omega_1\omega_2\cdots.
\end{equation}
It follows immediately that $P[Y_0=y_0]=1$ and
\begin{multline*}
P[Y_{n+1} = y_{n+1} ~|~ Y_n = y_n] = q_{ji}
= \frac{m_{ij}}{c_{j}}
= \begin{cases} 1/c(y_n), & \text{ if }y_{n+1}\in\sigma^{-1}(y_n) \\ 0, & \text{ otherwise} \end{cases}, \\
\text{ where } y_{n+1}\in[i],\, y_n\in[j].
\end{multline*}
In this way we recover the stochastic process $Y_0, Y_1,\ldots$ from Definition~\ref{def:rbt}. All we've done is construct one realization of an underlying probability space for this process.

Let us first prove (a). We work with cylinder sets, i.e.\ sets of the form $C=[i_0\cdots i_{m}] \subset \Sigma_M$ with $m\geq0$. The length of this cylinder is $m+1$. The \emph{reverse} of this cylinder is $\overline{C}=[i_m \cdots i_0] \subset \Sigma_{M^T}$ and is nonempty if and only if $C$ is nonempty in $\Sigma_M$. By applying~\eqref{P} several times we get $\nu(\overline{C})=\mu(C)$. We define the collection $\mathcal{E}_m$ to contain all unions of length $m+1$ cylinder sets, and the reverse of a set $B=\cup_\alpha C_\alpha \in \mathcal{E}_m$ is simply $\overline{B}=\cup_\alpha \overline{C_\alpha}$. Again, we get $\nu(\overline{B})=\mu(B)$. Moreover, the characteristic function $\mathds{1}_B$ is clearly integrable over $(\Sigma_{M^T},\nu)$. Applying Birkhoff's ergodic theorem to all these countably many characteristic functions at once, we get that
\begin{equation*}
W=\left\{\omega'=\omega'_0\omega'_1\omega'_2\cdots \in \Sigma_{M^T} ; \lim_{N\to\infty} \frac{1}{N} \sum_{n=0}^{N-1} \mathds{1}_{\overline{B}}(\sigma^n\omega')=\nu(\overline{B}) \text{ for all }B\in\mathcal{E}_m, \, m\geq1 \right\}
\end{equation*}
has full measure $\nu(W)=1$. Let $W_0$ be the intersection of $W$ with the cylinder set $[\omega_0]\subset \Sigma_{M^T}$. Since $\nu_0$ is just the (normalized) restriction of $\nu$ to $[\omega_0]$, we get also $\nu_0(W_0)=1$. Thus we can choose our random backward trajectory by choosing $\omega'\in W_0$ and setting $y_n=Y_n(\omega')$, $n=1,2,\ldots$. We need to show that this backward trajectory equidistributes for $\mu$.

Let $E$ be any open subset of $\Sigma_M$. For each $m\geq0$ let $E_m$ be the maximal element of $\mathcal{E}_m$ contained in $E$. Clearly $E_0 \subset E_1 \subset \cdots$ and since cylinder sets form a basis for the topology we have $E=\cup_m E_m$. Therefore $\mu(E)=\lim \mu(E_m)$. For $n\geq m$ we see by~\eqref{recover} that $y_n\in E_m$ if and only if $\sigma^{n-m}(\omega')\in \overline{E_m}$. Then we can calculate
\begin{multline*}
\liminf_{N\to\infty} \left(\frac{1}{N} \sum_{n=0}^{N-1} \delta_{y_n}\right)(E) \geq
\lim_{N\to\infty} \frac{1}{N} \sum_{n=m}^{N-1} \mathds{1}_{E_m}(y_n) =\\
=  \lim_{N\to\infty} \frac{1}{N}\sum_{n=0}^{N-m-1}\mathds{1}_{\overline{E_m}}(\sigma^n\omega')=\nu(\overline{E_m})=\mu(E_m).
\end{multline*}
Since this holds for arbitrary $m$, we get that the limes inferior is at least $\mu(E)$. Since $E$ was an arbitrary open set, this shows that $\frac{1}{N}\sum_{n=0}^{N-1}\delta_{y_n} \xrightarrow{weak-*} \mu$.

We now prove (b). Null recurrence of $Q$ means that our Markov chain $(\Sigma_{M^T},\nu_0)$ has the following property: with probability 1, the orbit under $\sigma$ of a randomly chosen point $\omega'=\omega'_0\omega'_1\cdots$ visits each cylinder $[i]$ infinitely often but with limiting frequency zero. By the Markov property, if there are infinitely many visits to $[i]$, then with probability $1$ each subcylinder $\overline{C}\subset [i]$ is also visited infinitely often. By the definition of $\nu_0$ we also get $\omega'_0=\omega_0$ with probability 1. Choose $\omega'$ with all of these properties and let $(y_n)=(Y_n(\omega'))$ be the corresponding backward trajectory of $y_0$. For each nonempty cylinder $C=[i_0\cdots i_m]\subset \Sigma_M$ there is $n\geq m$ with $\sigma^{n-m}(\omega')\in\overline{C}$, which gives $y_n\in C$. This shows the density in $\Sigma_M$ of our backward trajectory. Moreover, the visits of $y_n$ to each $[i]$ occur with the same limiting frequency as the visits of $\sigma^n(\omega')$ to $[i]$, and this frequency is zero.

Finally, we prove (c). Transience of $Q$ means that our Markov chain $(\Sigma_{M^T},\nu_0)$ has the following property: with probability 1, the orbit under $\sigma$ of a randomly chosen point $\omega'=\omega'_0\omega'_1\cdots$ visits each cylinder set $[i]$ only finitely many times. Proceeding as before, we see that the randomly chosen backward trajectory $(y_n)$ visits each $[i]$ only finitely many times.
\end{proof}

\begin{corollary}\label{cor:cms}
In the positive recurrent case, $\mu$ is ergodic and is the only fair measure on $\Sigma_M$. In the null recurrent and transient cases, there are no fair measures on $\Sigma_M$.
\end{corollary}

\begin{proof}
Suppose first that $Q$ is positive recurrent so that the fair measure $$\mu=\text{Markov}(\pi,P)$$ exists. We wish to show that $\mu$ is the only fair measure. By Lemma~\ref{lem:erg-dec}, it suffices to show that each ergodic fair measure $\mu'$ is equal to $\mu$. Let $C=[i_0\cdots i_m]\subset \Sigma_{M}$ be a cylinder set. By Theorem~\ref{th:back} (a) applied to $\mathds{1}_C$ we can find a point $y_0$ from the $\mu'$-full measure set whose almost every backward trajectory visits $C$ with limiting frequency $\mu'(C)$. But by Theorem~\ref{th:cms} (a), the random backward trajectory $(y_n)$ of $y_0$ equidistributes for $\mu$. Since the cylinder set $C$ is both closed and open, we get $\mu(\partial C)=\mu(\emptyset)=0$. So equidistribution tells us that $(y_n)$ visits $C$ with limiting frequency $\mu(C)$. In this way we get equality $\mu(C)=\mu'(C)$ for all cylinder sets, from which it follows that $\mu'=\mu$.

Next we show the ergodicity of $\mu$. This follows immediately from Lemma~\ref{lem:erg-dec} and the fact that $\mu$ is the only fair measure. Alternatively, ergodicity follows because $\mu$ is a Markov measure whose transition matrix $P$ is positive recurrent.

Next, we consider what happens when $Q$ is null recurrent or transient. We wish to show that there are no fair measures. By Lemma~\ref{lem:erg-dec}, it suffices to show that there are no ergodic fair measures. Suppose to the contrary that $\mu'$ is an ergodic fair measure. There must be some cylinder set $[i]$ with $\mu'([i])>0$. By Theorem~\ref{th:back} (a) applied to $\mathds{1}_{[i]}$ we can find a point $y_0$ from the $\mu'$-full measure set whose almost every backward trajectory visits $[i]$ with limiting frequency $\mu'([i])$. But this contradicts Theorem~\ref{th:cms}.
\end{proof}

\begin{corollary}\label{cor:cms2}
In the positive recurrent case, for each $y_0\in\Sigma_M$, for a random choice $(y_n)$ of the backward trajectory,
\begin{equation}\label{limJac}
\lim_{n\to\infty} \sqrt[n]{c(y_0)\cdots c(y_{n-1})} = \exp \int \log(c)\,d\mu = \exp \sum_{i\in\I} \pi_i \log c_i,
\end{equation}
where all three expressions may be infinite. If additionally $- \sum \pi_i \log(\pi_i) < \infty$, then 
\begin{equation}\label{Jacfair}
\int \log(c)\,d\mu = h_\mu(\sigma) = -\sum_{i,j\in\I} \pi_i p_{ij} \log(p_{ij}) = \hfair(\sigma) < \infty.
\end{equation}
\end{corollary}

\begin{proof}
Theorem~\ref{th:back}~(c) applied to the measure $\mu$ gives~\eqref{limJac} for $\mu$-almost every $y_0$. In particular, in each cylinder set $[i]$ we get the result for at least one point $y'_0$. Now let $y_0$ be any other point in the same cylinder $[i]$. There is a natural way to identify backward trajectories of $y_0$ and $y'_0$. In the language of the probability model $(\Sigma_{M^T},\nu_0)$ developed earlier, we identify $(y_n)$ with $(y'_n)$ if they both arise from the same point $\omega'\in W_0$. But then $c(y_n)=c(y'_n)$ for all $n$. This shows that~\eqref{limJac} holds not just almost everywhere, but in fact for every $y_0\in\Sigma_M$.

The partition by length 1 cylinder sets $\X=\{[i] ; i\in \I \}$ is a one-sided generating partition and the condition $- \sum \pi_i \log(\pi_i) < \infty$ just says that the (Shannon) entropy $H_\mu(\eta)$ of this partition is finite. Then by Theorem~\ref{th:back}~(d) we get the first equality in~\eqref{Jacfair}. The sum in~\eqref{Jacfair} is just the well-known formula for the entropy of a Markov measure. This is also the fair entropy, since $\mu$ is the unique fair measure. Finally, since $\eta$ is a finite-entropy generating partition, the (Kolmogorov-Sinai) entropy $h_\mu(\sigma)$ is less than or equal to $H_\mu(\X)$ and is therefore finite, see~\cite[Equation (9.1.15) and Corollary 9.2.5]{VO}.
\end{proof}


We now show some examples on known shift spaces. Later we will associate them with interval maps too.

\vspace{0.1in}

\begin{example}\label{ex:null} (Null recurrent case)
We consider a classic unbiased random walk on $\mathbb Z $, where we can go 1 step forward or 1 step backward from each state. Here is the transition diagram:
\begin{center}
\begin{tikzcd}
{\cdots} &[-30]
{\circ} \arrow[r, bend left=25]
& \circ 
 \arrow[r, bend left=25]
\arrow[l, bend left=25]
& \circ
\arrow[r, bend left=25]
\arrow[l, bend left=25]
& \circ
\arrow[l, bend left=25]
&[-30pt]\cdots
\end{tikzcd}
\end{center}
\vspace{.5em}
We calculate the transition matrix $M$ in the standard way and by (\ref{Q}) we get the corresponding stochastic matrix $Q$.

\begin{equation*}
M =
\begin{bmatrix}
 \ddots &  \vdots & \vdots & \vdots & \vdots & \vdots & \reflectbox{$\ddots$} \\
\cdots& 0 & 1 & 0 & 0 &0 & \cdots \\
\cdots&1 & 0 & 1 & 0 &0 & \cdots \\
\cdots &0& 1 & 0 & 1 & 0 & \cdots \\
\cdots &0&  0& 1 & 0 & 1  & \cdots \\
\cdots &0&  0& 0& 1 & 0 &  \cdots \\
\reflectbox{$\ddots$} &  \vdots & \vdots & \vdots & \vdots& \vdots &\ddots
\end{bmatrix}
,\qquad
\renewcommand{\arraystretch}{1.2}
Q=
\begin{bmatrix}
 \ddots &  \vdots & \vdots & \vdots & \vdots & \vdots & \reflectbox{$\ddots$} \\
\cdots& 0 & \frac12 & 0 & 0 &0 & \cdots \\
\cdots& \frac12 & 0 & \frac12 & 0 &0 & \cdots \\
\cdots &0& \frac12 & 0 & \frac12 & 0 & \cdots \\
\cdots &0&  0& \frac12 & 0 & \frac12  & \cdots \\
\cdots &0&  0& 0& \frac12 & 0 &  \cdots \\
\reflectbox{$\ddots$} &  \vdots & \vdots & \vdots & \vdots& \vdots &\ddots
\end{bmatrix}.
\end{equation*}

We already know from probability theory (see \cite{Fel}) that 
 $Q$ is null recurrent, so by Corollary \ref{cor:cms} there is no fair measure in this case. Nevertheless $ \sqrt[n]{c(y_0)\cdots c(y_{n-1})} \rightarrow 2$ as $n \rightarrow \infty.$
%
%

\end{example}

\begin{example}\label{ex:transient} (Transient case)
A biased random walk on $\mathbb Z $  can be defined as the option to go 2 steps forward or 1 step backward from any state. Here is the transition diagram:
\begin{center}
\begin{tikzcd}
{\cdots} \arrow[rr,  dashed, bend left=25]
&[-30]
{\circ} \arrow[rr, bend left=25]
& \circ 
 \arrow[rr, bend left=25]
\arrow[l, bend left=25]
& \circ
\arrow[rr, bend left=25]
\arrow[l, bend left=25]
& \circ
\arrow[rr, bend left=25]
\arrow[l, bend left=25]
& \circ
\arrow[rr,  dashed, bend left=25]
\arrow[l, bend left=25]
& \circ
\arrow[l, bend left=25]
&[-30pt]\cdots
\end{tikzcd}
\end{center}
\vspace{.5em}
The matrices $M$ and $Q$ can be found the same way as before.
\begin{equation*}
M =
\begin{bmatrix}
 \ddots &  \vdots & \vdots & \vdots & \vdots & \vdots & \vdots & \reflectbox{$\ddots$} \\
\cdots& 0& 0 & 1 & 0 & 0 &0 & \cdots \\
\cdots&1& 0 & 0 & 1 & 0 &0 & \cdots \\
\cdots &0& 1 & 0& 0 & 1 & 0 & \cdots \\
\cdots &0&  0& 1 & 0& 0 & 1  & \cdots \\
\cdots &0&  0& 0& 1 & 0& 0 &  \cdots \\
\reflectbox{$\ddots$} &  \vdots & \vdots & \vdots & \vdots & \vdots& \vdots &\ddots
\end{bmatrix}
,\qquad
\renewcommand{\arraystretch}{1.2}
Q=
\begin{bmatrix}
 \ddots &  \vdots & \vdots & \vdots & \vdots & \vdots & \vdots & \reflectbox{$\ddots$} \\
\cdots& 0 & \frac12 & 0 & 0 & 0 &0 & \cdots \\
\cdots& 0 & 0 & \frac12 & 0 & 0 &0 & \cdots \\
\cdots & \frac12 & 0  & 0 & \frac12 & 0  & 0 & \cdots \\
\cdots &0 & \frac12& 0  & 0 & \frac12 & 0  & \cdots \\
\cdots &0&  0& \frac12& 0  & 0 & \frac12 &  \cdots \\
\reflectbox{$\ddots$} &  \vdots & \vdots& \vdots & \vdots & \vdots& \vdots &\ddots
\end{bmatrix}.
\end{equation*}
Then $\left(Q^{3n}\right)_{00}= \dfrac{(3n)!}{(2n)!n!}\cdot\dfrac{1}{2^{3n}}$, and so $\sum_{n\in \mathbb N} \left(Q^{n}\right)_{00} < \infty$ which by \cite[pg. 389]{Fel} shows us that the stochastic matrix is  transient and therefore by Corollary \ref{cor:cms} there is no fair measure.

Another classic shift example is the unbiased random walk on $\mathbb Z^3$. We will get similar results. The stochastic matrix  $Q$ is transient and so again by Corollary \ref{cor:cms} there is no fair measure in this case.

\end{example}

\begin{example} (Positive recurrent case)
For our  last example on shift spaces we choose a process which can be defined as an option to go anywhere from the origin or 1 step backward from any other state, as is shown in the following transition diagram and matrices:
\begin{center}
\begin{tikzcd}
 \circ 
 \ar[loop,out=150,in=210,distance=30]
  \arrow[r, bend left=20]
    \arrow[rr, bend left=20]
     \arrow[rrr, bend left=20]
      \arrow[rrrr, bend left=20]
& \circ 
\arrow[l, bend left=20]
& \circ 
\arrow[l, bend left=20]
& \circ 
\arrow[l, bend left=20]
& \circ 
\arrow[l, bend left=18]
&[-30pt]\cdots
\end{tikzcd}
\end{center}
\begin{equation*}
M =
\begin{bmatrix}
& 1& 1 & 1 & 1 & 1 &1 & \cdots \\
&1& 0 & 0 & 0 & 0 &0 & \cdots \\
&0& 1 & 0& 0 & 0 & 0 & \cdots \\
&0&  0& 1 & 0& 0 & 0  & \cdots \\
&0&  0& 0& 1 & 0& 0 &  \cdots \\
 & \vdots & \vdots & \vdots & \vdots & \vdots& \vdots &\ddots
\end{bmatrix}
,\qquad
\renewcommand{\arraystretch}{1.2}
Q=
\begin{bmatrix}

&  \frac12 &  \frac12 & 0 & 0 & 0 &0 & \cdots \\
&  \frac12 & 0 & \frac12 & 0 & 0 &0 & \cdots \\
&  \frac12 & 0 & 0 & \frac12 & 0 &0 & \cdots \\
&  \frac12 & 0 & 0 & 0 & \frac12 &0 & \cdots \\
&  \frac12 & 0 & 0 & 0 & 0 &\frac12 & \cdots \\
&  \vdots & \vdots& \vdots & \vdots & \vdots& \vdots &\ddots
\end{bmatrix}.
\end{equation*}
%
%

The reader can easily check that
$\pi= \begin{bmatrix} \frac{1}{2} \, \frac{1}{4} \cdots \frac{1}{2^i} \cdots \end{bmatrix}$
satisfies~\eqref{pi}. Therefore $Q$ is positive recurrent and there is a unique fair measure. We can also calculate $P$ using~\eqref{P}:
\begin{equation*}
P =
\begin{bmatrix}
& \frac12& \frac14 & \frac18 & \frac{1}{16}  & \cdots & \frac{1}{2^i}  & \cdots \\
&1& 0 & 0 & 0 & \cdots  &0 & \cdots \\
&0& 1 & 0& 0 & \cdots  & 0 & \cdots \\
&0&  0& 1 & 0&\cdots  & 0  & \cdots \\
&0&  0& 0& 1 & \cdots & 0 &  \cdots \\
 & \vdots & \vdots & \vdots & \vdots & \vdots& \vdots &\ddots
\end{bmatrix}
\end{equation*}
And so the measure $\textnormal{Markov}(\pi,P)$ is the fair measure. 
 Finally, we may calculate the fair entropy as the entropy of the measure $\textnormal{Markov}(\pi,P)$ on the shift space, using the well-known formula for the entropy of a Markov measure,

\begin{equation*}
\hfair(f)=-\sum_{ij} \pi_i p_{ij} \log p_{ij} = \log 2.
\end{equation*}
This is the same as the Gurevich entropy of $(\Sigma_M, \sigma)$, and so the fair measure is the maximal measure.
\end{example}

\section{Isomorphisms for Fair Measures}\label{sec:iso}

We would like to extend our results from shift spaces to various maps on the interval, dendrites, etc. through the use of Markov partitions. But before we can proceed, we need to develop a notion of isomorphism. Our notion is inspired by the isomorphisms modulo small sets in~\cite{Hof}. For a system $(X,f)$, a set $B\subset X$ is called \emph{totally invariant} if $f^{-1}(B)=B$.

\begin{definition}
Two systems $(X_1,f_1)$, $(X_2, f_2)$ are called \emph{isomorphic modulo countable invariant sets} if there exist totally invariant countable sets $N_1\subset X_1$, $N_2\subset X_2$ and a bijection $\psi:X_1\setminus N_1 \to X_2\setminus N_2$, bimeasurable with respect to the Borel $\sigma$-algebras, such that $f_2\circ\psi=\psi\circ f_1$.
\end{definition}

From the point of view of fair measures, deleting countable totally invariant sets is rather harmless. For the points that remain, the tree of preimages is unchanged, so that the meaning of a random backward orbit is the same as before. And countable sets have measure zero for every non-atomic measure. Letting $\Mna(\cdot,\cdot)$ denote the set of non-atomic invariant Borel probability measures for a system, we get bijections
\begin{equation*}
\Mna(X_i,f_i) \to \Mna(X_i\setminus N_i,f_i), \quad i=1,2,
\end{equation*}
given by restriction $\mu \mapsto \mu|_{X_i\setminus N_i}$, as well as a bijection
\begin{equation*}
\Mna(X_1\setminus N_1,f_1) \to \Mna(X_2\setminus N_2, f_2)
\end{equation*}
given by composition $\mu \mapsto \mu \circ \psi^{-1}$.
If $\mu_1\in\Mna(X_1,f_1)$ and $\mu_2\in\Mna(X_2,f_2)$ correspond to each other under these bijections, then $\psi:(X_1,f_1,\mu_1) \to (X_2,f_2,\mu_2)$ is a conjugacy in the sense of measure theory (where null sets are negligible). Consequently,
\begin{itemize}
\item the entropies are equal $h_{\mu_1}=h_{\mu_2}$,
\item $\mu_1$ is ergodic if and only if $\mu_2$ is, and
\item The Jacobians are related by $J_{\mu_1} = J_{\mu_2} \circ \psi$.
\end{itemize}
But total invariance of our deleted sets gives $c_1\circ f_1 = c_2 \circ f_2 \circ \psi$ on $X_1\setminus N_1$, i.e. $\mu_1$-almost everywhere. In light of Lemma~\ref{lem:Jac} we can conclude that if $\mu_2$ is fair, then so is $\mu_1$. By the symmetry of the situation, the reverse implication holds also. In particular, we have proved

\begin{theorem}\label{th:isom}
A Borel isomorphism modulo countable invariant sets induces an entropy-preserving bijection of non-atomic fair measures. Moreover, this implies that the two systems have the same fair entropy.
\end{theorem}

We remark that in general, fairness of a measure is not an invariant of measure theoretic conjugacy. This is because adding or deleting a set of measure zero can change the function $c\circ f$ almost everywhere, cf. Lemma~\ref{lem:Jac}. Figure~\ref{fig:1} illustrates how this might happen.

\begin{figure}[htb!!]
\vspace{1em}
\scalebox{.95}{
\includegraphics[width=.3\textwidth]{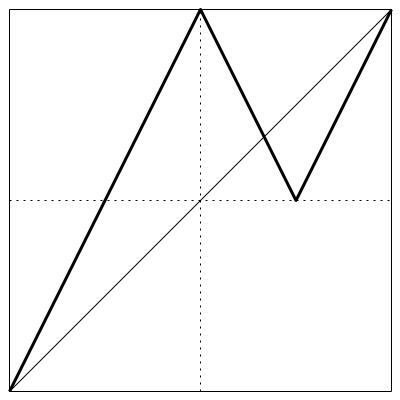}
\begin{picture}(0,0)
\put(0,90){\scriptsize 3 preimages}
\put(0,30){\scriptsize 1 preimage}
\put(-133,-2){$\underbrace{\hspace{5.2em}}_{\substack{\text{zero}\\\text{measure}}}$}
\put(-72,-2){$\underbrace{\hspace{5.2em}}_{\substack{\text{normalized}\\\text{Lebesgue measure}}}$}
\put(-100,135){\small (a) Not fair}
\end{picture}
\hspace{.2\textwidth}
\includegraphics[width=.3\textwidth]{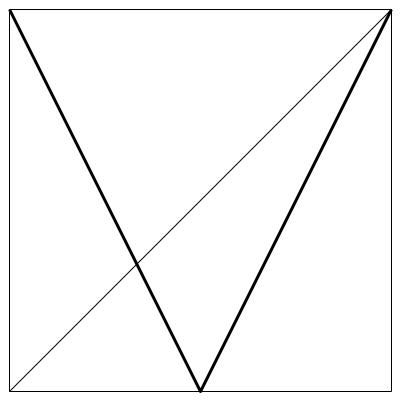}
\begin{picture}(0,0)
\put(0,60){\scriptsize 2 preimages}
\put(-131,-2){$\underbrace{\hspace{10.4em}}_{\text{Lebesgue measure}}$}
\put(-90,135){\small (b) Fair}
\end{picture}
}
\\[1.5em]
\caption{Deleting the (measure zero) left half of the interval and rescaling gives a measure theoretic conjugacy, but system (a) is not fair while system (b) is fair.}\label{fig:1}
\end{figure}
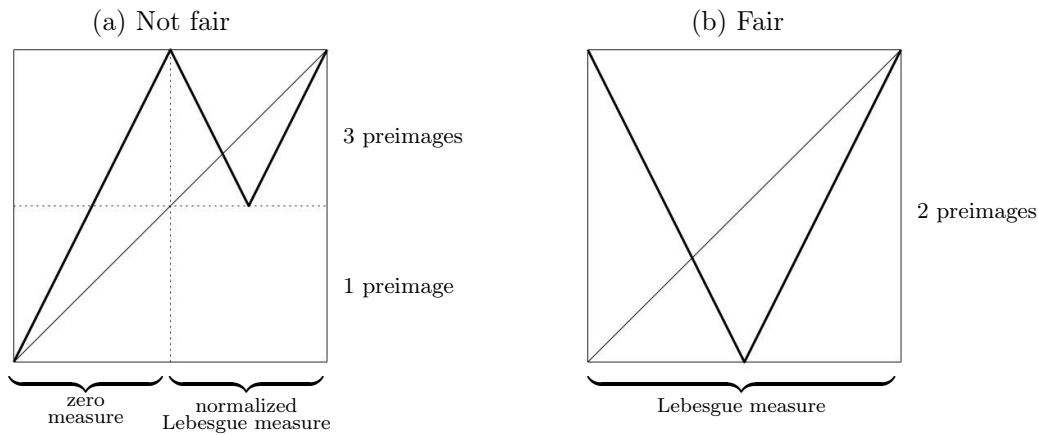

Fair measures with atoms are not addressed in Theorem~\ref{th:isom}, but are rather simple to understand. If $(X,f)$ has a totally invariant periodic orbit, then we can equidistribute point masses along this orbit and we obtain a purely atomic ergodic fair measure. Conversely, we observe that

\begin{proposition}\label{prop:atoms}
All the atoms of a fair measure belong to totally invariant periodic orbits.
\end{proposition}
\begin{proof}
As a general fact regarding invariant probability measures, atoms can only occur in periodic orbits and with each point of a periodic orbit $P$ having the same mass $m$. If $P$ is not totally invariant, then at least one point $x\in P$ has $c(x)\geq2$. Then the preimage of $x$ in $P$ has mass under a fair measure equal to both $m$ and $\frac{m}{c(x)}$, which implies that $m=0$.
\end{proof}

\section{Countably Markov and mixing Interval Maps}\label{LIM}

Throughout this section our space is the unit interval $X=[0,1]$, our map $f$ is allowed to have countably many pieces of continuity and monotonicity, and our partition $\X$ is Markov. That means that
\begin{itemize}
\item $\X$ consists of open intervals $(a,b)$ and singletons $\{x\}$. We write $$\I=\left\{i=(a,b);~(a,b)\in\X\right\}, \quad C=\left\{x;~\{x\}\in\X\right\}$$
for the sets of \emph{partition intervals} and \emph{partition points}.
\item For each $i\in\I$, the restriction $f|_i$ is continuous and strictly monotone.
\item For each pair $i,j\in\I$, either $f(i)\supset j$ or else $f(i)\cap j=\emptyset$.
\item The partition points form a forward invariant set $f(C)\subset C$.
\end{itemize}
We assume additionally that $f$ is \emph{mixing}, i.e. for each pair $U,V$ of nonempty open sets there is $N\geq0$ such that for all $n\geq N$, $f^n(U)\cap V\neq\emptyset$. If all these assumptions are satisfied, then we call $f$ \emph{countably Markov and mixing}.

A \emph{homterval} for an interval map $f$ is a nonempty open interval $U\subset[0,1]$ such that for all $n\geq0$, $f$ maps $f^n(U)$ homeomorphically onto $f^{n+1}(U)$.
\begin{lemma}
A countably Markov and mixing interval map has no homtervals.
\end{lemma}
\begin{proof}
A mixing interval map has at least one \emph{critical point} $x\in(0,1)$, such that $f$ is not monotone on any neighborhood of $x$ . Let $U$ be any nonempty open interval. By the mixing property applied to $U$, $V=[0,x)$, and $V'=(x,1]$, there is a common value $n$ such that $f^n(U)$ meets both $V$ and $V'$. If $U$ were a homterval, then by the intermediate value theorem $f^n(U)$ would be a neighborhood of $x$, but then $f$ could not map $f^n(U)$ homeomorphically onto $f^{n+1}(U)$.
\end{proof}

To a countably Markov and mixing interval map $f$ we associate a \emph{transition matrix} $M$ with rows and columns indexed by $\I$ and entries
\begin{equation}
\label{mij}
m_{ij}=\begin{cases}
0,&\text{ if }f(i)\cap j=\emptyset\\
1,&\text{ if }f(i)\supset j
\end{cases}.
\end{equation}
Using the mixing property of $f$ it is easy to see that $M$ is irreducible. Thus we may form the shift space $\Sigma_M$ and look for fair measures as in Section~\ref{sec:cms}.


All of our results about Markov interval maps flow out of the following isomorphism theorem.

\begin{theorem}\label{th:cmm}
Let $f$ be countably Markov and mixing with transition matrix $M$. Then $(\Sigma_M,\sigma)$ and $([0,1],f)$ are isomorphic modulo countable invariant sets.
\end{theorem}

Before beginning the proof we record one more or less standard lemma.
\begin{lemma}\label{lem:symbolic}
Let $f$ be countably Markov and mixing with transition matrix $M$.
\begin{enumerate}[label=(\alph*)]
\item\label{lem:symbolic_a}
For each nonempty cylinder set $[i_0\cdots i_n]\subset \Sigma_M$, the set $$U=i_0 \cap f^{-1}(i_1) \cap \cdots \cap f^{-n}(i_n)$$ is a nonempty open interval mapped homeomorphically by $f^n$ onto $i_n$.
\item\label{lem:symbolic_b}
If $[i'_0 \cdots i'_n]\neq[i_0 \cdots i_n]$ is another nonempty cylinder set, then the corresponding set $U'$ is disjoint from $U$.
\end{enumerate}
\end{lemma}
\begin{proof}
\ref{lem:symbolic_a}: By induction in the length of the cylinder. The base case $n=0$ is clear. Given $[i_0\cdots i_n]$, the induction hypothesis applied to $[i_1\cdots i_n]$ gives us a nonempty open interval $V=i_1 \cap f^{-1}(i_2) \cap \cdots f^{-n+1}(i_n)$ contained in $i_1$ and mapped homeomorphically by $f^{n-1}$ onto $i_n$. But $f(i_0) \supset i_1$ and $f|_{i_0}$ is continuous and strictly monotone. We get that $U=i_0\cap f^{-1}(V)=f|_{i_0}^{-1}(V)$ is a nonempty open interval mapped by $f$ homeomorphically onto $V$, and the result follows.

\ref{lem:symbolic_b}: If $i_j\neq i'_j$, then the sets $f^j(U)\subset i_j$ and $f^j(U')\subset i'_j$ are disjoint. Therefore $U, U'$ are disjoint.
\end{proof}

\begin{proof}[Proof of Theorem~\ref{th:cmm}]
Elements of $\Sigma_M$ are called \emph{itineraries}. Given an itinerary $\omega=i_0 i_1 \cdots$ we put $U_n(\omega)=i_0\cap f^{-1}(i_1)\cap \cdots \cap f^{-n}(i_n)$ for each $n\geq0$. We have a nested sequence of open intervals $U_0(\omega) \supset U_1(\omega) \supset \cdots$ as well as a nested sequence of closed intervals $\overline{U_0(\omega)} \supset \overline{U_1(\omega)} \supset \cdots$. Taking the intersection, we obtain a non-empty closed interval $\bigcap_{n=0}^\infty \overline{U_n(\omega)}$. This interval must be degenerate (a singleton), for otherwise its interior is a homterval for $f$. In this way we get a map
\begin{equation}\label{psi}
\psi:\Sigma_M\to[0,1], \quad \psi(\omega)=x\text{, where }\{x\}=\bigcap_{n=0}^\infty \overline{U_n(\omega)}.
\end{equation}
We say that $\omega$ is an itinerary of the point $x=\psi(\omega)$. If in~\eqref{psi} the closures are not needed so that $x\in\bigcap_{n=0}^\infty U_n(\omega)$, then we call $\omega$ a \emph{true itinerary}; otherwise we call $\omega$ a \emph{false itinerary}.

True itineraries behave well. If $\omega=i_0i_1\cdots$ is a true itinerary of the point $x$, then $f^n(x)\in i_n$ for all $n\geq0$. In particular, $x$ does not belong to the set of \emph{pre-critical points} $D=\bigcup_{n=0}^\infty f^{-n}(C)$. Conversely, each point $x\in[0,1]\setminus D$ has a true itinerary given by
\begin{equation}\label{phi}
\phi:[0,1]\setminus D \to \Sigma_M, \quad \phi(x)=i_0i_1i_2\cdots\text{, where }f^n(x)\in i_n\in\I \text{ for all }n\geq0.
\end{equation}
Even better, if $\omega$ is a true itinerary for $x$, we see immediately that $\sigma(\omega)$ is a true itinerary for $f(x)$. Thus, we get the conjugacy relation $\psi(\sigma(\omega))=f(\psi(\omega))$ for all true itineraries $\omega$.

Conversely, suppose $\omega=i_0i_1\cdots$ is an itinerary whose shift $\sigma(\omega)$ is true. Write $y=\psi(\sigma(\omega))$. Since $i_0\cap f^{-1}(i_1)$ is mapped homeomorphically by $f$ onto $i_1$, we find a point $x\in i_0\cap f^{-1}(y)$. Since $y$ was not a pre-partition point, neither is $x$, and we see immediately from~\eqref{phi} that $\omega$ is a true itinerary for $x$. We have shown that an itinerary $\omega$ is true if and only if $\sigma(\omega)$ is true.

False itineraries behave much worse. If $\omega=i_0 i_1 \cdots$ is a false itinerary of the point $x$, then there is a minimal number $n\geq0$ such that $x\notin U_n(\omega)$, called the \emph{order} of the false itinerary. Recall that $U_n(\omega)$ is an open interval contained in $i_0$ and mapped homeomorphically by $f^n$ onto $i_n$. Since $x\in\overline{U_n(\omega)}$ and $f|_{i_0}$ is continuous and $x\in i_0$ (assuming $n\geq1$), we get $f^n(x)\in\overline{i_n}$. Therefore $f^n(x)$ is an endpoint of $i_n$ (this holds also in the case $n=0$). In particular, $f^n(x)\in C$, so $x\in D$. Thus, false itineraries are only associated with pre-partition points. The conjugacy relation may not hold: if $\omega$ is a false itinerary for $x$ of order $0$, then $\sigma(\omega)$ is still a false itinerary but perhaps not for $f(x)$. This is because we have no control over the value $f(x)$ when $x\in C$ is not a continuity point for $f$.

A pre-partition point $x\in D$ may have zero, one, or two false itineraries, but not more. For let $n\geq0$ be minimal such that $f^n(x)\in C$. We have already seen that each false itinerary $\omega=i_0i_1\cdots$ for $x$ has order $n$. Thus $x\in U_{n-1}(\omega)$ and $x$ is an endpoint of $U_n(\omega)$. Because of the nesting $\overline{U_n(\omega)} \supset \overline{U_{n+1}(\omega)} \supset \cdots$, we see that $x$ is an endpoint of $U_j(\omega)$ for all $j\geq n$. By Lemma~\ref{lem:symbolic}~\ref{lem:symbolic_b}, the condition $x\in U_{n-1}(\omega)$ uniquely determines the symbols $i_0, \cdots, i_{n-1}$. There are at most two choices for $i_n$, namely, the at most two partition intervals with common endpoint $f^n(x)$. Now let $\omega'=i'_0i'_1\cdots$ be another false itinerary. We claim that if $i_0\cdots i_n=i'_0\cdots i'_n$, then $\omega=\omega'$. This follows inductively, for if $i_0\cdots i_j=i'_0\cdots i'_j$ with $j\geq n$ but $i_{j+1}\neq i'_{j+1}$, then $U_{j+1}(\omega)$ and $U_{j+1}(\omega')$ are disjoint subintervals of $U_j(\omega)=U_j(\omega')$, and all three of these intervals have $x$ as an endpoint, which is impossible.

Let $N\subset \Sigma_M$ denote the set of false itineraries. We have shown so far that $N=\psi^{-1}(D)$ and that $\phi,\psi,f,\sigma$ are related by the following commutative diagram:
\begin{equation}\label{cd}
\begin{tikzcd}[row sep=1cm, column sep=1.5cm]
\Sigma_M\setminus N \arrow[r, "\sigma"] \arrow[d, shift left=-1, "\psi" swap] & \Sigma_M\setminus N \arrow[d, shift left=-1, "\psi" swap] \\
\left[0,1\right]\setminus D \arrow[r, "f"] \arrow[u, shift left=-1, "\phi" swap] & \left[0,1\right]\setminus D \arrow[u, shift left=-1, "\phi" swap]
\end{tikzcd}
\end{equation}
We still need to show that
\begin{enumerate}[label=(\roman*)]
\item\label{it:cd1} $N$ and $D$ are countable and totally invariant, and
\item\label{it:cd2} $\psi:\Sigma_M\setminus N \to [0,1]\setminus D$ and its inverse $\phi$ are both measurable with respect to the Borel $\sigma$-algebras.
\end{enumerate}

We start with~\ref{it:cd1}. Since $f$ is at most countable-to-one and $C$ is countable, it follows that $D=\cup_{n=0}^\infty f^{-n}(C)$ is countable. Backward invariance $f^{-1}(D)\subset D$ is clear from construction, and forward invariance $f(D)\subset D$ is inherited from $C$. Therefore $D$ is totally invariant. Since $N=\psi^{-1}(D)$ and $\psi$ is at most two-to-one, it follows that $N$ is countable. We already showed that an itinerary $\omega$ is true if and only if $\sigma(\omega)$ is true. Thus, the set of false itineraries $N$ is also totally invariant.

To prove~\ref{it:cd2}, it suffices to show continuity of the maps $\psi:\Sigma_M\to[0,1]$ and $\phi:[0,1]\setminus D \to \Sigma_M$. To see that $\psi$ is continuous at a given point $\omega=i_0i_1\cdots$, let $\epsilon>0$ be given. Since the nested intersection in~\eqref{psi} contains only the one point $x=\psi(\omega)$, it follows that there is $n\geq0$ such that $\overline{U_n(\omega)}$ is contained in the $\epsilon$-neighborhood of $x$. But the cylinder set $[i_0\cdots i_n]$ is an open neighborhood of $\omega$ mapped by $\psi$ into $\overline{U_n(\omega)}$.

To see that $\phi$ is continuous consider any nonempty cylinder set $[i_0\cdots i_n]\subset\Sigma_M$ and define $U$ as in Lemma~\ref{lem:symbolic}~\ref{lem:symbolic_a}. Then $U$ is open in $[0,1]$, so $U\setminus D$ is an open subset of $[0,1]\setminus D$ in the subspace topology. But $U\setminus D=\phi^{-1}([i_0\cdots i_n])$. We have shown that each cylinder set has an open preimage under $\phi$, and since cylinder sets form a basis for the topology on $\Sigma_M$, we are done.
\end{proof}


Now we explore the consequences of our isomorphism theorem for our interval map.

\begin{theorem}\label{th:cmmfair}
Let $f$ be countably Markov and mixing with transition matrix $M$. Form $Q$ using~\eqref{cj} and~\eqref{Q}. Let $\Acc(C)$ denote the set of accumulation points of $C$.
\begin{enumerate}[label=(\alph*)]
\item\label{th:cmmfair_a} If $Q$ is positive recurrent, then
    $f$ has exactly one non-atomic fair measure $\mu$. It is ergodic and has full support.
    For any non-partition point $y_0\in[0,1]\setminus C$, a randomly chosen backward trajectory $(y_n)$ equidistributes for $\mu$, and the geometric averages $\sqrt[n]{c(y_0)\cdots c(y_{n-1})}$ converge to $\exp h_\mu(f)$.
\item\label{th:cmmfair_b} If $Q$ is null recurrent, then
    $f$ does not have any non-atomic fair measure.
    For any non-partition point $y_0\in[0,1]\setminus C$, a randomly chosen backward trajectory $(y_n)$ is dense in $[0,1]$, but visits each set $E$ bounded away from $\Acc(C)$ with limiting frequency zero.
\item\label{th:cmmfair_c} If $Q$ is transient, then
    $f$ does not have any non-atomic fair measure.
    For any non-partition point $y_0\in[0,1]\setminus C$, a randomly chosen backward trajectory $(y_n)$ converges to $\Acc(C)$.
\end{enumerate}
\end{theorem}
\begin{proof}
Most of these results follow immediately from Theorem~\ref{th:cms} and its corollaries together with our isomorphism result Theorem~\ref{th:cmm}. It is also critical to note that the maps $\phi,\psi$ in~\eqref{cd} are continuous. Therefore if $U\subset[0,1]$ is open, then there is an open set $V\subset\Sigma_M$ whose symmetric difference with $\psi^{-1}(U)$ is contained in the countable invariant set $N$. This gives us the full support result in~\ref{th:cmmfair_a}, because the fair measure $\text{Markov}(\pi,P)$ on $\Sigma_M$ has full support. It also gives us the density result in~\ref{th:cmmfair_b} because $y_n\in U$ if and only if $\phi(y_n)\in V$.

Since $C$ is closed and countable, so is its set of accumulation points $\Acc(C)$. Now if $E\subset[0,1]$ is bounded away from this set, i.e. $\overline{E}\cap\Acc(C)=\emptyset$, then $E$ is contained in a finite union of partition intervals $i_1\cup \cdots \cup i_n$. Then $y_n$ cannot visit $E$ unless $\phi(y_n)$ is in the corresponding union of cylinders $[i_1]\cup\cdots\cup[i_n]$. This gives us the frequency of visits to $E$ in part~\ref{th:cmmfair_b}. It also gives us the convergence of $y_n$ to $\Acc(C)$ in part~\ref{th:cmmfair_c}, where convergence means that the distance between $y_n$ and the nearest point of $\Acc C$ goes to zero.
\end{proof}

\section{Lebesgue fair models.}

\begin{definition}
An interval map is called \emph{Lebesgue fair} if Lebesgue measure is fair for it. A \emph{Lebesgue fair model} for an interval map $f$ is a Lebesgue fair map $g$ conjugate to $f$ by a monotone increasing homeomorphism $\phi:[0,1]\to[0,1]$, $g\circ\phi=\phi\circ f$.
\end{definition}


The main result of this section is a construction of the Lebesgue fair models for countably Markov and mixing interval maps. In a sense, it allows us to visualize the fair measures which we found. We start with two easy lemmas.

\begin{lemma}\label{lem:fmfm}
The Lebesgue fair models for $f$ are in bijective correspondence with the non-atomic fair measures of full support.
\end{lemma}
\begin{proof}
Let $\mu$ be such a measure and define $\phi_\mu:[0,1]\to[0,1]$ by $x\mapsto\mu([0,x])$. This is a monotone increasing homeomorphism with $\mu\circ\phi_\mu^{-1}$ equal to the Lebesgue measure $\Leb$. Thus $\phi_\mu : ([0,1],f,\mu) \to ([0,1],g,\Leb)$ is a measure theoretic isomorphism everywhere, i.e., without the removal of measure zero sets. Therefore $\lambda$ is fair for $g$ (see Section~\ref{sec:iso}).

Conversely, let $g$ be a Lebesgue fair model for $f$ with conjugating homeomorphism $\phi$. Then $\mu_\phi=\Leb\circ\phi$ is a Borel probability measure and $\phi:([0,1],f,\mu_\phi)\to([0,1],g,\Leb)$ is a measure theoretic isomorphism everywhere. It follows that $\mu_\phi$ is a non-atomic fully supported fair measure for $f$. 

Since the operations $\mu\mapsto\phi_\mu$ and $\phi\mapsto\mu_\phi$ are clearly inverse to each other, we've found the desired bijective correspondence.
\end{proof}


\begin{lemma}\label{lem:determined}
A homeomorphism $g:(a,b)\to(c,d)$ is uniquely determined by its orientation (increasing or decreasing) and the Jacobian for Lebesgue measure, provided that Lebesgue measure is non-singular for $g$.
\end{lemma}
\begin{proof}
Assume $g$ is an increasing homeomorphism. Given $x\in(a,b)$, put $B=(a,x)$ so that $g(B)=(c,g(x))$. By~\eqref{Jac} we get $g(x)=c+\int_B J\,d\Leb$, where $J$ is the Jacobian and $\Leb$ is the Lebesgue measure.
\end{proof}

Now let $f$ be countably Markov and mixing with transition matrix $M$. If $f$ falls into case~\ref{th:cmmfair_b} or~\ref{th:cmmfair_c} of Theorem~\ref{th:cmmfair}, then there are no Lebesgue fair models. But in case~\ref{th:cmmfair_a}, we see that there is a unique Lebesgue fair model $g$, and we would like to know what it looks like.

Here is a construction for the graph of $g$. On the horizontal axis we draw the sets $\phi(i\cap f^{-1}(j))$, $i,j\in\X$, $i\cap f^{-1}(j)\neq\emptyset$, where $\phi=\phi_\mu$ is the conjugacy given by Lemma~\ref{lem:fmfm} using the unique non-atomic fair measure $\mu$ for $f$. Thankfully, there is no need to calculate $\phi$ explicitly; it suffices to calculate $\pi$ and $P$ from~\eqref{pi} and~\eqref{P}. We know that $\phi(i\cap f^{-1}(j))$ is an open interval if $i,j\in\I$ and a singleton otherwise. We also know that these sets form a partition. And we know exactly where to draw these sets in $[0,1]$ because we know their ordering ($\phi$ preserves ordering) and their lengths (denoted $\len$) $$\len(\phi(i\cap f^{-1}(j)))=\mu(i\cap f^{-1}(j)) = \begin{cases}\pi_i p_{ij},&\text{if }i,j\in\I\\0,&\text{otherwise}\end{cases}.$$

On the vertical axis we draw the partition sets $\phi(j)$, $j\in\X$. We get open intervals when $j\in\I$ and singletons otherwise. Again, we know exactly where to draw these sets because we know their ordering and their lengths $$\len(\phi(j))=\mu(j)=\begin{cases}\pi_j,&\text{if }j\in\I\\0,&\text{otherwise}\end{cases}.$$

Now we fill in the graph of $g$. If $\phi(i\cap f^{-1}(j))$ is a singleton, then $g$ maps this point to the single element of $\phi(j)$. If $\phi(i\cap f^{-1}(j))$ is an interval, then $g$ carries this interval homeomorphically onto $\phi(j)$. The Jacobian here for the fair (non-singular) Lebesgue measure is the constant $c_j=\#f^{-1}(x)$, $x\in j$. So by Lemma~\ref{lem:determined}, this piece of $g$ is affine with slope $\pm c_j$, with the plus or minus determined by the orientation of $f|_i$. Notice that this slope agrees with the lengths of these intervals, because $$\frac{\pi_j}{\pi_i p_{ij}} = \frac{\pi_j}{\pi_j q_{ji}} = \frac{c_j}{m_{ij}}=c_j.$$

We may summarize our construction as a theorem.

\begin{theorem}\label{th:fairmodel}
Let $f$ be a countably Markov and mixing interval map whose associated Markov shift $\Sigma_M$ has a fair measure $\textnormal{Markov}(\pi,P)$. Then the unique Lebesgue fair model for $f$ is the piecewise affine map $g$ defined as follows: For each pair $i,j\in\X$ with $i\cap f^{-1}(j)\neq\emptyset$,
if $i,j\in\mathcal{I}$, then $g$ maps $(x,x')$ affinely onto $(y,y')$ with the same orientation as $f|_i$, and
if $j=\{c\}$ is a singleton, then $g$ maps $x$ to $y$, where
\begin{gather*}
x=\sum_{\mathclap{(k,l)\in W_{ij}}} \pi_k p_{kl}, \qquad x'=x+\pi_i p_{ij}, \qquad y=\sum_{l\in W_j}\pi_l, \qquad y'=y+\pi_j,\\
W_{ij}=\left\{ (k,l)\in\I\times\I ;~ \emptyset \neq k\cap f^{-1}(l) \text{ lies to the left of }i\cap f^{-1}(j)\right\}\\
W_j=\left\{ l\in\I ;~ l \text{ lies to the left of }j\right\}.
\end{gather*}
\end{theorem}

In this way, we have constructed the Lebesgue fair model $g$ for $f$ using only combinatorial information.
Incidentally, this shows that there is quite a lot of flexibility in the definition of the function $f$ -- as long as we keep the right Markov structure and the mixing property, we automatically get a topological conjugacy to the same Lebesgue fair model $g$.

\vspace{0.2in}

\begin{example}
Applying the results of section \ref{LIM}, especially Theorem \ref{th:cmmfair}, we can associate interval maps with shift spaces. The next example is just one of a pile of transitive mappings which are a lot like a random walk. We choose one which has all partition intervals the same length. One partition interval will cover itself and another 4 partition intervals when the map there is increasing or another 2 intervals when it decreases.  Here are an exact formula for $f(x)$, a picture, and also the corresponding matrices $M$ and $Q$ with the main diagonals shown in bold:

\begin{minipage}[c][0.55cm][c]{.55\textwidth} 
\begin{equation*}\label{53}
f(x) = \begin{cases}
5x-8n-2, & {\rm if} ~ x \in I_{2n},\\\\
-3x+8n+6, & {\rm if} ~ x \in I_{2n+1} ,\\
\end{cases}
\end{equation*}\\[1em]
\centering
where $n\in \mathbb Z$ and $I_k= \langle k, k +1\rangle$.
\end{minipage}%
\begin{minipage}[c][6cm][c]{.45\textwidth}
\includegraphics[width=.7\textwidth]{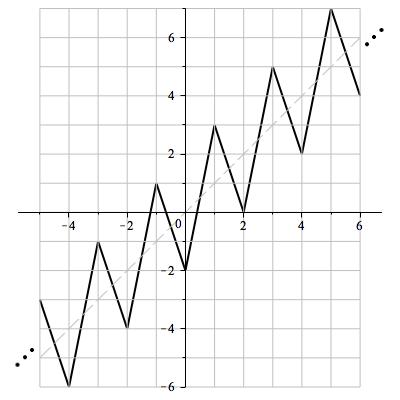}
\end{minipage}
\vspace{-1em}
\begin{equation*}
{\small
\setcounter{MaxMatrixCols}{20} 
M =
\begin{bmatrix}
 \ddots &  \vdots & \vdots & \vdots & \vdots &  \vdots  & \vdots & \vdots  & \reflectbox{$\ddots$} \\
  \cdots& \mathbf{1} & 1 & 1 & 0 & {0} & 0 & 0   & \cdots \\
 \cdots& 1 & \mathbf{1} & 1 & 0 & {0} & 0 & 0   & \cdots \\
\cdots& 1 & 1 & \mathbf{1} & 1 & {1} & 0 & 0   & \cdots \\
\cdots& 0 & 0 & 1 & \mathbf{1} & {1} & 0 & 0   & \cdots \\
\cdots & 0 & 0 & 1 & 1 & \mathbf{1} & 1 & 1 &  {\cdots} \\

\cdots & 0 & 0& 0 & 0 & {1} & \mathbf{1} & 1   & \cdots \\
\cdots & 0 & 0& 0 & 0 & {1} & 1 &\mathbf{1}   &  \cdots \\
\reflectbox{$\ddots$} & \vdots & \vdots & \vdots & \vdots& \vdots& \vdots  & \vdots  &\ddots
\end{bmatrix}
,\quad
\renewcommand{\arraystretch}{1.2}
Q=
\begin{bmatrix}
 \ddots &  \vdots & \vdots & \vdots & \vdots & \vdots & \vdots & \vdots & \vdots & \reflectbox{$\ddots$} \\
\cdots& \frac15& \frac15 &  \mathbf{\frac15}  & \frac15 & {\frac15} & 0 & 0 & 0 & \cdots \\
\cdots & 0 & 0 & \frac13 &\mathbf{\frac13}  & {\frac13} & 0 & 0 & 0& \cdots \\
\cdots & 0 & 0  & \frac15 & \frac15 & \mathbf{\frac15} & \frac15 & \frac15 & 0 & \cdots \\
\cdots &  0 & 0 & 0  & 0 & {\frac13}  & \mathbf{\frac13}  & \frac13 & 0 & \cdots \\
\cdots & 0 & 0 & 0 & 0 & {\frac15} & \frac15 &  \mathbf{\frac15}  & \frac15  &\cdots \\
\cdots &  0& 0& 0 & 0  & {0} & 0 & \frac13 & \mathbf{\frac13}  & \cdots \\
\reflectbox{$\ddots$} &  \vdots & \vdots & \vdots& \vdots & \vdots & \vdots & \vdots& \vdots &\ddots
\end{bmatrix}.
}
\end{equation*}
The map $f$ has no totally invariant periodic orbits, so if there are any fair measures, they must be non-atomic (see Proposiiton~\ref{prop:atoms}).
It is not so hard to check that $\pi = (\cdots 3 \, 5 \, 3\,  5\,  3 \cdots)$, where $\pi_{2i} = 5, \pi_{2i+1} = 3$, satisfies the formula $\pi Q=\pi$. But even if we rescale $\pi,$ $ \sum_{i\in \mathbb Z} \pi_i = \infty$. Therefore, our $\pi$ cannot satisfy formula (\ref{pi}) and  by \cite[Section XV.11]{Fel} there is no summable solution, and so the stochastic matrix is not positive recurrent and there is no fair measure for $f$.

\end{example}

\begin{example}\label{ex:BruinTodd}
Bruin and Todd studied the thermodynamic formalism for a countably piecewise linear interval map $f_{\lambda}:[0,1] \rightarrow [0,1]$, as defined in~\eqref{BT} below. To be complete, at the partition points $C=\{0,1,\lambda,\lambda^2,\ldots\}$ we make $f_\lambda$ continuous from the left, and we put $f_\lambda(0)=0$. All choices of the parameter $\lambda\in(0,1)$ give a map from the same topological conjugacy class. There is no measure of maximal entropy, and the topological entropy (supremum of entropies of invariant probability measures) is equal to $\log 4$, see~\cite{BT}. \\[-1em]
\begin{minipage}[c][6cm][c]{.6\textwidth} 
\begin{equation}\label{BT}
f_{\lambda}(x) = \begin{cases}
\dfrac{x- \lambda}{1- \lambda}, & {\rm if} ~ x \in i_1,\\[1.5em]
\dfrac{x- \lambda^n}{\lambda(1- \lambda)}, & {\rm if} ~ x \in i_n, n \geq 2,\\
\end{cases}
\end{equation}\\[.5em]
\centering
where $i_n = ( \lambda^n, \lambda^{n-1})$
\end{minipage}%
\begin{minipage}[c][6cm][c]{.4\textwidth}
\scalebox{.9}{
\includegraphics[width=.8\textwidth]{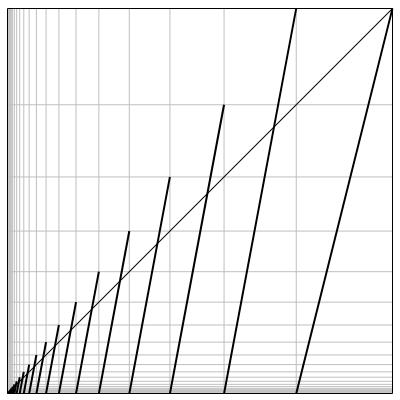}
\begin{picture}(0,0)
\put(-30,-10){$W_1$}
\put(-60,-10){$W_2$}
\put(-90,-10){$\cdots$}
\end{picture}
}
\end{minipage}

We want to calculate the non-atomic fair measures for this map. Therefore we write down the transition matrix $M$ and the corresponding stochastic matrix $Q$,
\begin{equation*}
M =
\begin{bmatrix}
1 & 1 & 1 & 1 & \cdots \\
1 & 1 & 1 & 1 & \cdots \\
0 & 1 & 1 & 1 & \cdots \\
0 & 0 & 1 & 1 & \cdots \\
0 & 0 & 0 & 1 & \cdots \\
\vdots & \vdots & \vdots & \vdots & \ddots
\end{bmatrix}
,\qquad
\renewcommand{\arraystretch}{1.2}
Q=
\begin{bmatrix}
\tfrac12 & \tfrac12 & 0 & 0 & 0 & \cdots \\
\tfrac13 & \tfrac13 & \tfrac13 & 0 & 0 & \cdots \\
\tfrac14 & \tfrac14 & \tfrac14 & \tfrac14 & 0 & \cdots \\
\tfrac15 & \tfrac15 & \tfrac15 & \tfrac15 & \tfrac15  & \cdots \\
\vdots & \vdots & \vdots & \vdots & \vdots & \ddots
\end{bmatrix}
.
\end{equation*}
The reader can easily check that
$\pi=\frac1e \begin{bmatrix} \frac{1}{0!} \frac{1}{1!} \frac{1}{2!} \frac{1}{3!} \cdots \end{bmatrix}$
satisfies~\eqref{pi}. Therefore $Q$ is positive recurrent and our interval map $f_\lambda$ has a unique non-atomic fair measure. We may calculate $P$ using~\eqref{P} and the Lebesgue fair model $g$ using Theorem~\ref{th:fairmodel}.

\begin{minipage}[c][6cm][c]{.3\textwidth} 
\begin{equation*}
\renewcommand{\arraystretch}{1.2}
P=
\begin{bmatrix}
\tfrac12 & \tfrac13 & \tfrac18 & \tfrac{1}{30} & \cdots \\
\tfrac12 & \tfrac13 & \tfrac18 & \tfrac{1}{30} & \cdots \\
0 & \tfrac23 & \tfrac28 & \tfrac{2}{30} & \cdots \\
0 & 0 & \tfrac68 & \tfrac{6}{30} & \cdots \\
0 & 0 & 0 & \tfrac{24}{30} & \cdots \\
\vdots & \vdots & \vdots & \vdots & \ddots
\end{bmatrix}
\end{equation*}
\end{minipage}\hspace{.1\textwidth}%
\begin{minipage}[c][6cm][c]{.26\textwidth}
\includegraphics[width=\textwidth]{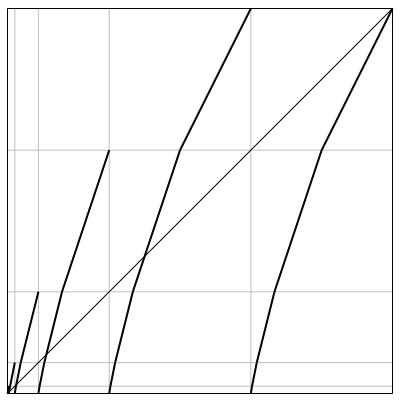}
\end{minipage}\hspace{.04\textwidth}%
\begin{minipage}[c][6cm][c]{.26\textwidth}
\centering
\small
The Lebesgue fair model. Notice that\\
$|g'(x)|=\#g^{-1}(g(x))$
Lebesgue almost everywhere.
\end{minipage}

Finally, we may use Corollary~\ref{cor:cms2} to calculate the fair entropy as the entropy of the measure $\textnormal{Markov}(\pi,P)$ on the shift space $\Sigma_M$,

\begin{equation*}
\hfair(f)=-\sum_{ij} \pi_i p_{ij} \log p_{ij} \approx \log(2.85053).
\end{equation*}
\end{example}

\section{Maps on Tame Graphs}

In this section we move beyond the interval and look for fair measures for graph maps. Our main theorem is for tame graphs, a generalization of finite graphs introduced in~\cite{BBPRV}. We start by recalling the definition.

A \emph{continuum} is a nonempty, compact, connected metric space. By $E(G)$ we denote the \emph{endpoints} of the continuum $G$, i.e. the points $x\in G$ having arbitrarily small neighborhoods $V$ with one-point boundaries $\#\partial V=1$. Similarly, by $B(G)$ we denote the \emph{branching points}, i.e. the points $x\in G$ such that any sufficiently small neighborhood $V$ of $x$ has at least three points in its boundary $\#\partial V\geq 3$. A continuum $G$ is called a \emph{tame graph} if the set $E(G)\cup B(G)$ has a countable closure.

An arc $\alpha$ in a continuum $G$ is called a \emph{free arc} if the set $\alpha^\circ = \alpha \setminus E(\alpha)$ is open in $G$. A \emph{tame partition} for $G$ is a countable family $\P$ of free arcs with pairwise disjoint interiors covering $G$ up to a countable set of points. In~\cite{BBPRV} it was shown that a continuum $G$ is a tame graph if and only if it admits a tame partition, which happens if and only if all but countably many points of $G$ have a neighborhood in $G$ which is a finite graph. It was also shown that every tame graph is locally connected (and thus a Peano continuum).

Let $g:G\to G$ be a continuous map on a tame graph. Suppose that there is a tame partition $\P$ such that
\begin{itemize}
\item For every $i\in\P$ the restriction $g|_i : i \to g(i)$ is a homeomorphism, and
\item For every pair $i,j\in\P$ if $g(i)\cap j^\circ\neq\emptyset$, then $g(i)\supset j$.
\end{itemize}
Then we will call $\P$ a \emph{countable Markov partition} for $g$. As a reminder, $g$ is called \emph{mixing} if for each pair of nonempty open sets $U,V\subset G$ there is $N\geq0$ such that for all $n\geq N$, $g^n(U)\cap V\neq\emptyset$.

Our main idea for studying tame graph maps is to cut up the graph into arcs and glue those arcs back together to get an interval. The resulting interval map will be discontinuous, but for our purposes it does not matter much.

\begin{definition}\label{def:cutandpaste}
Let $g:G\to G$ be a continuous mixing map of a tame graph with countable Markov partition $\P$. Let $\psi:\bigcup_{i\in\P}i^\circ \to [0,1]$ and $f:[0,1]\to[0,1]$ be maps such that
\begin{itemize}
\item $\psi\left(\bigcup_{i\in\P}i^\circ\right)$ has a countable complement in $[0,1]$,
\item For each $i\in\P$, $\psi$ maps $i^\circ$ homeomorphically onto its image,
\item For distinct $i,j\in\P$, $\psi(i^\circ)\cap\psi(j^\circ)=\emptyset$ (thus $\psi$ is injective), and
\item $f(x)=(\psi\circ g\circ\psi^{-1})(x)$ if this composition of maps is defined at $x$, and otherwise $f(x)\not\in\psi\left(\bigcup_{i\in\P}i^\circ\right)$.
\end{itemize}
Then the system $([0,1],f)$ will be called a \emph{cut-and-paste model} for $(G,g)$.
\end{definition}

\begin{lemma}\label{lem:exists-cap}
If $G$ is a tame graph and $g:G\to G$ is a continuous mixing map with a countable Markov partition, then there exists a cut-and-paste model for $(G,g)$.
\end{lemma}
\begin{proof}
Let $\P$ denote the countable Markov partition. It may be finite or countably infinite; we give the proof for the countably infinite case. Choose an enumeration $i_1, i_2, i_3, \ldots$ of the partition arcs. Define $\psi$ on $i_n^\circ$ to be any homeomorphism of $i_n^\circ$ onto $(2^{-n}, 2^{-n+1})$. Finally, define $f$ by
\begin{equation*}
f(y)=\begin{cases}
0, &\text{if }y=2^{-n} \text{ for some } n\geq0,\\
0, &\text{if }g(\psi^{-1}(y)) \text{ is not in the interior of any partition arc},\\
\psi(g(\psi^{-1}(y))), & \text{otherwise}.
\end{cases}
\end{equation*}
Then $([0,1],f)$ is a cut-and-paste model for $(G,g)$.
\end{proof}

\begin{theorem}\label{th:iso-cap}
If $([0,1],f)$ is a cut-and-paste model for $(G,g)$, then they are isomorphic modulo countable invariant sets.
\end{theorem}
\begin{proof}
Let $\P$, $\psi$ be as in Definition~\ref{def:cutandpaste}. Let $U_1$ denote the open set $\bigcup_{i\in\P}i^\circ$ and $C_1$ its complement in $G$. Let $N_1=\bigcup_{n=0}^\infty g^{-n}(C_1)$. Now $C_1$ is countable by the definition of a tame partition and is forward-invariant under $g$ by~\cite[Lemma 2.5(v)]{BBPRV}. Since $g$ is an at most countable-to-one map, $N_1$ is countable and totally invariant. Similarly, let $U_2$ denote the open set $\bigcup_{i\in\P} \psi(i^\circ)=\psi(U_1)$ and let $C_2$ denote its complement in $[0,1]$. Let $N_2=\bigcup_{n=0}^\infty f^{-n}(C_2)$. Applying Definition~\ref{def:cutandpaste}, we see that $f$ is also at most countable-to-one, and that $N_2$ is also countable and totally invariant.

Consider the restricted map $\psi: G\setminus N_1 \to [0,1]\setminus N_2$. We wish to show that $\psi$ is well-defined, bijective, and bimeasurable. Then the identity $\psi\circ g = f\circ\psi$ clearly follows, so that $\psi$ is the desired isomorphism.

\emph{(Well-defined):} Let $x\in G\setminus N_1$. We must show that $\psi(x)\not\in N_2$. We have $g^n(x)\in U_1$ for all $n\geq0$. It follows inductively that $f^n(\psi(x))=\psi(g^n(x))\in U_2$ for all $n\geq0$.

\emph{(Bijective):} Injectivity is free from Definition~\ref{def:cutandpaste}. To prove surjectivity, choose $y\in[0,1]\setminus N_2$. Then $f^n(y)\in U_2$ for each $n\geq0$, so that $\psi^{-1}$ is defined at each point along the forward orbit of $y$. It follows inductively that $f^n(y)=\psi\circ g^n\circ\psi^{-1}(y)$ for all $n\geq0$. Therefore $g^n(\psi^{-1}(y))\in U_1$ for all $n\geq0$, that is, $\psi^{-1}(y)\in G\setminus N_1$.

\emph{(Bimeasurable):} In fact, we show that both $\psi, \psi^{-1}$ are continuous at every point where they are defined. For if $y=\psi(x)$, then there is $i\in\P$ with $x\in i^\circ$. Then $i^\circ$ is an open neighbourhood of $x$ in $G$, $\psi(i^\circ)$ is an open neighbourhood of $y$ in $[0,1]$, and $\psi$ gives a homeomorphism between these two neighbourhoods.
\end{proof}

Given a tame graph $G$ and a continuous, mixing map $g:G\to G$ with a countable Markov partition $\P$, we may define the \emph{transition matrix} $M=M(g,\P)=(m_{ij})_{i,j\in\P}$ by the rule
\begin{equation*}
m_{ij}=\begin{cases}
1, &\text{if }g(i)\supset j,\\
0, &\text{if }g(i)\cap j^\circ=\emptyset.
\end{cases}
\end{equation*}

\begin{theorem}
Let $G$ be a tame graph, $g:G\to G$ a continuous mixing map with countable Markov partition $\P$, and let $M$ be the associated transition matrix. Then $(G,g)$ is isomorphic modulo countable invariant sets to the shift space $(\Sigma_M,\sigma)$.
\end{theorem}
\begin{proof}
Consider the cut-and-paste model $([0,1], f)$ and the map $\psi$ constructed in Lemma \ref{lem:exists-cap}. Then $f$ is a countably Markov and mixing interval map (in the sense of Section~\ref{LIM}) with respect to the partition intervals $\I=\{\psi(i^\circ \cap g^{-1}(j^\circ));~i,j\in\P\}$. The transition matrix for $f$ is the same as the transition matrix for $g$ with respect to the refined tame partition $\P\vee g^{-1}(\P) = \{i\cap g^{-1}(j);~i,j\in\P\}$, which is again a countable Markov partition for $g$~\cite[Lemma 2.5]{BBPRV}. Combining Theorems~\ref{th:cmm} and \ref{th:iso-cap} we see that $(G,g)$ is isomorphic modulo countable invariant sets with $(\Sigma_{M(g,\P\vee g^{-1}(\P))}, \sigma)$. But this system is in turn topologically conjugate to $(\Sigma_M,\sigma)$, being nothing more than its higher block presentation with blocks of length 2 (see \cite[Section 1.4]{K}).
\end{proof}

\begin{example} Figure~\ref{fig:dendrite} illustrates a countably Markov and mixing map $g$ on a dendrite $G$ sometimes called the star or the locally connected fan. $G$ is the union in $\mathbb{R}^2$ of countably many line segments $A_i, i=1,2, \ldots$ called blades. Each blade has one endpoint at the origin; the other endpoint is called the tip of the blade. No blade contains any other and the lengths of the blades converge to zero. Each blade is subdivided into arcs at countably many points converging to the tip of the blade -- this defines the Markov partition $\P$. The map $g$ fixes the origin. If we denote the sequence of points subdividing $A_i$ as $(x^i_n)_{n=0}^\infty$, ordering them along $A_i$ from the origin to the tip, then $g$ maps $x^i_{2n}$ to the origin and $x^i_{2n+1}$ to the tip of blade $A_{\max(1,i-1)+n}$ for all $n$, and is piecewise affine between these partition points. In the left part of the figure, the label next to a pair of partition arcs indicates which blade those arcs will be mapped onto; in the right part of the figure, the labels name the blades. Continuity of $g$ is clear. The topological mixing property is not clear, but can be ensured by an appropriate choice of the lengths of the blades and of the subarcs into which they are partitioned -- we omit the calculations. It is also fairly easy to show that $g$ has no totally invariant periodic orbits, and thus no atomic fair measures.

\begin{figure}[htb!!]
\includegraphics[width=.8\textwidth]{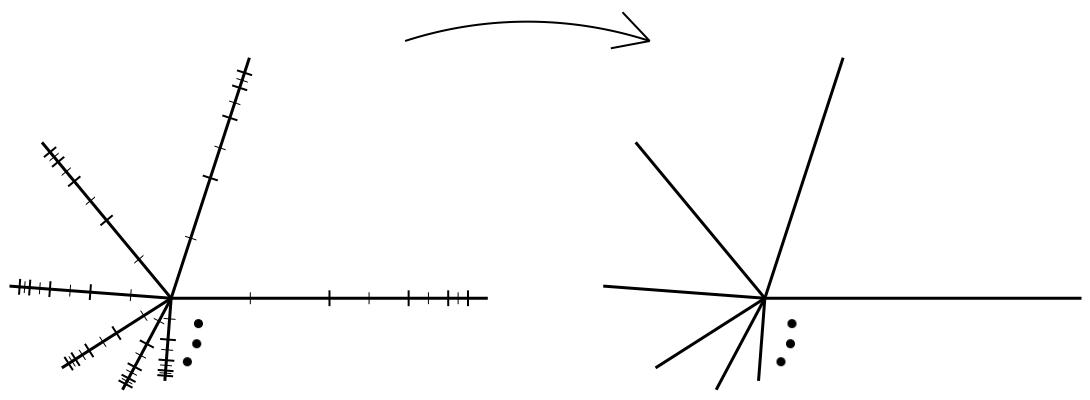}
\begin{picture}(0,0)
\put(-276,23){\tiny $A_1$}
\put(-239,23){\tiny $A_2$}
\put(-217,23){\tiny $\cdots$}
\put(-287,48){\tiny $A_1$}
\put(-278,76){\tiny $A_2$}
\put(-271,90){\rotatebox{73}{\tiny $\cdots$}}
\put(-307,52){\rotatebox{-45}{\tiny $A_2$}}
\put(-321,70){\rotatebox{-45}{\tiny $A_3$}}
\put(-330,82){\rotatebox{-45}{\tiny $\cdot\!\cdot\!\cdot$}}
\put(-315,38){\rotatebox{-5}{\tiny $A_3$}}
\put(-333,40){\rotatebox{-5}{\tiny $A_4$}}
\put(-344,41){\rotatebox{-5}{\tiny $\cdot\!\cdot\!\cdot$}}
\put(-4,29){\small $A_1$}
\put(-84,113){\small $A_2$}
\put(-157,87){\small $A_3$}
\put(-176,36){\small $A_4$}
\put(-156,4){\small $A_5$}
\put(-132,-6){\small $A_6$}
\put(-115,-4){\small $A_7$}
\put(-185,125){\small $g$}
\end{picture}
\caption{A dendrite map on the star dendrite $G$.}
\label{fig:dendrite}
\end{figure}
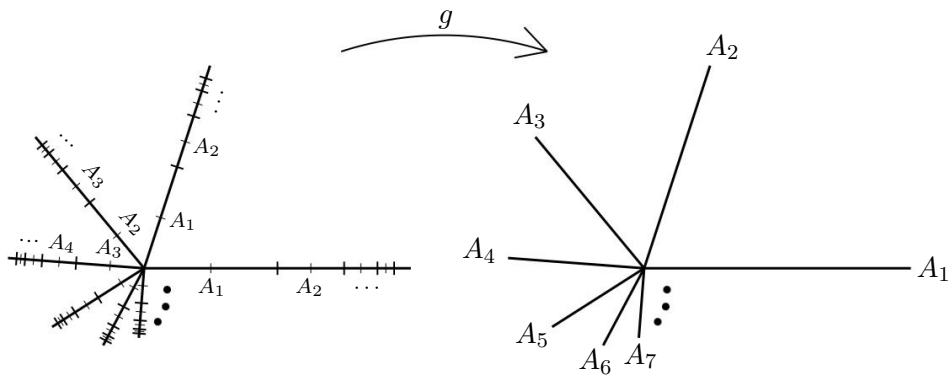

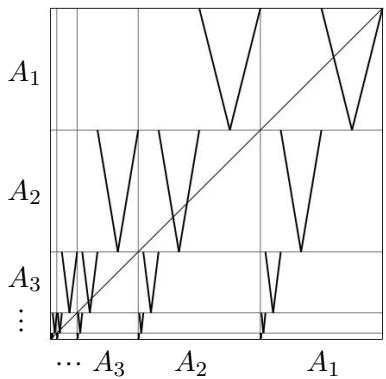
\begin{figure}[htb!!]
\includegraphics[width=.3\textwidth]{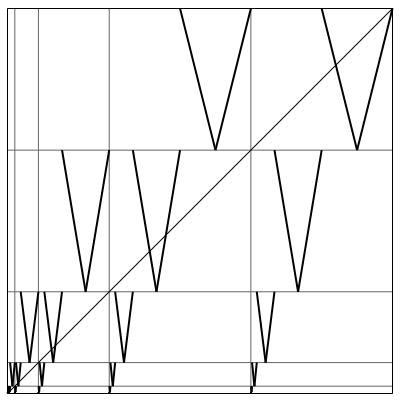}
\begin{picture}(0,0)
\put(-35,-10){\small $A_1$}
\put(-85,-10){\small $A_2$}
\put(-116,-10){\small $A_3$}
\put(-130,-10){\small $\cdot\!\cdot\!\cdot$}
\put(-148,100){\small $A_1$}
\put(-148,55){\small $A_2$}
\put(-148,22){\small $A_3$}
\put(-145,4){\rotatebox{90}{\small $\cdot\!\cdot\!\cdot$}}
\end{picture}
\caption{A cut-and-paste model for $(G,g)$.}
\label{fig:cutandpaste}
\end{figure}

A cut-and-paste model $([0,1],f)$ for $(G,g)$ is shown in Figure~\ref{fig:cutandpaste} -- the labels show how the blades of $G$ have been placed within the interval $[0,1]$. Notice that the dendrite map and the interval map have the same symbolic dynamics. Moreover, the interval map $f$ is Lebesgue fair -- it is the same as the Lebesgue fair model from Example~\ref{ex:BruinTodd}, but with each affine piece of the graph replaced by two pieces with twice the slope. We conclude that $(G,g)$ has a unique fair measure. Moreover, we may calculate the fair entropy via the Rohlin formula as $\int_0^1 \log|f'(x)| dx \approx \log(2.85053)+\log(2)$ -- this just adds $\log(2)$ to the fair entropy from Example~\ref{ex:BruinTodd}, which makes sense heuristically, since each point has twice as many preimages.

\end{example}


\end{document}